\documentclass[12pt,a4paper]{amsart}
\usepackage{amsmath,amssymb,a4}
\usepackage{enumerate,amssymb}
\theoremstyle{plain}
\newtheorem{thm}{Theorem}[section]
\newtheorem{lemma}[thm]{Lemma}
\newtheorem{rmk}[thm]{Remark}
\newtheorem{corollary}[thm]{Corollary}

\newtheorem{example}[thm]{Example}
\newtoks\prt
\newtheorem{proclaim}[thm]{\the\prt}
\theoremstyle{definition}

\def\eqn#1$$#2$${\begin{equation}\label#1#2\end{equation}}

\numberwithin{equation}{section}

\def\epsilon{\varepsilon}

\def\N{\mathbb N}

\def\loc{\operatorname{loc}}
\def\rn{\mathbb R^n}
\def\spt{\operatorname{supp}}

\def\Ln{\mathcal L^n}
\def\H{\mathcal {H}}
\newtoks\by
\newtoks\paper
\newtoks\book
\newtoks\jour
\newtoks\yr
\newtoks\pages
\newtoks\vol
\newtoks\publ
\def\ota{{\hbox\vol{???}}}
\def\cLear{\by=\ota\paper=\ota\book=\ota\jour=\ota\yr=\ota
\pages=\ota\vol=\ota\publ=\ota}
\def\endpaper{\the\by, {\the\paper},
\the\jour, \the\yr, \the\vol , \the\pages.\cLear}
\def\endbook{\the\by, \the\book, \the\publ.\cLear}
\def\endprep{\the\by, \the\paper, \the\jour.\cLear}
\def\endyearprep{\the\by, \textit{\the\paper}, \the\jour, (\the\yr).\cLear}
\def\name#1#2{#2 #1}

\def\sptt{\operatorname{supp}}
\def\loc{loc}

\def\C{\mathcal{C}}
\def\R{\mathbb {R}}
\def\Rn{\mathbb {R}^n}

\def\applim{\operatorname{app}\ \lim}

\title{Composition operator for functions of bounded variation}
\author{Lud\v{e}k Kleprl\'{i}k}
\address{Department of Mathematics and Statistics, University of
Jyv\"askyl\"a, P.O. Box 35 (MaD), FIN-40014, Jyv\"askyl\"a, Finland}
\email{\tt  kleprlud@fit.cvut.cz}
\thanks{The author was supported by the Academy of Finland, project number 263363.}

\begin{document}
\begin{abstract}
We study the optimal conditions on a homeomorphism $f:\Omega\subset \R^n\to \R^n$ to guarantee that the composition $u\circ f$ belongs to the  space  of functions of bounded variation  for every function $u$ of bounded variation. We show that a sufficient and necessary condition is the existence of a constant $K$ such that $|Df|(f^{-1}(A))\leq K\Ln(A)$ for all Borel sets $A$. We also characterize homeomorphisms which maps sets of finite perimeter to sets of finite perimeter.  Towards these results we study when $f^{-1}$ maps sets of measure zero onto sets of measure zero (i.e. $f$ satisfies the Lusin $(N^{-1})$ condition).
\end{abstract}

\maketitle

\section{Introduction}
In this paper we address the following issue. Suppose that $\Omega\subset \Rn$ is an open set, $f : \Omega \to \rn$ is a homeomorphism and a function of bounded variation and  $u$ is a function of $BV(f(\Omega))$. Under
which conditions can we then conclude that $u \circ  f \in BV(\Omega)$ or that $u \circ  f$ is weakly differentiable in some weaker sense? Our main theorem gives a complete answer to this question.
\prt{Theorem}
\begin{proclaim}\label{homos}
Let $\Omega_1, \Omega_2$ be open subsets of $\R^n$  and let $f\in BV_{\loc}(\Omega_1,\Omega_2)$ be a homeomorphism. Suppose that there is a constant $K>0$ such that
\eqn{klic1}
$$|Df|(f^{-1}(A))\leq K \Ln (A)\text{ for all Borel sets } A\subset\Omega_2.$$
Then the operator $T_f(u)=u\circ f$ maps  functions from $BV(\Omega_2)$  into $BV(\Omega_1)$ and
\eqn{spojitost1}
$$|D(u\circ f)|(\Omega_1)\leq K |Du|(\Omega_2).$$
On the other hand, if $f$ is a homeomorphism of $\Omega_1$ onto $\Omega_2$ such that  the operator $T_f$  maps $C_0(\Omega_2)\cap BV(\Omega_2)$  into $BV(\Omega_1)$, 
then $f\in BV_{\loc}(\Omega_1,\Omega_2)$ and there exists a constant $K>0$ such that \eqref{klic1}  holds.
\end{proclaim}
 The class of homeomorphisms  that satisfy \eqref{klic1} forms a natural extension of a special class of mappings of finite distortion. More precisely: in the fourth chapter we show  that the set of homeomorphisms in $W^{1,1}_{\loc}$ with the property  \eqref{klic1}  coincides with the known class of homeomorphisms with finite distortion satisfying that there exists a constant $K>0$ such that
$$|Df(x)|\leq K|J_f(x)| \text{ for a.e. } x \in \Omega.$$
It is known that for this class of Sobolev homeomorphisms we have $T_f(u):=u\circ f\in W^{1,1}$ for all $u\in W^{1,1}$. See \cite{GGR} or \cite{VU} for details. Hence naturally $T_f$ maps function from $W^{1,1}$ to $BV$. 

Let us note that the morphism property of $T_f$ on $BV$ was also known under the assumption that the homeomorphism $f$ belongs to class of mappings with a Lipschitz inverse. This can be found in \cite[Theorem 3.16]{AFP}, or \cite{He2}. We show that the above two classes of homeomorphisms differ (and our class contains both of them).

To prove Theorem \ref{homos} we need to know that $f$ satisfies the Lusin $(N^{-1})$ condition, i.e. preimages of sets of Lebesgue measure zero have measure zero. If the condition fails then there is a set $A\subset \Omega_1$ such that $\Ln(A)>0$ and $\Ln\bigl(f(A)\bigr)=0$. 
Then we can redefine $u$ on the  null set $f(A)$ arbitrarily and the composed function may fail to be  measurable. 
On the other hand, if $f$ satisfies the Lusin $(N^{-1})$ condition then the validity of our statement for one representative of $u$ implies the validity for all representatives, because the compositions only differ
 on a set of measure zero. The Lusin $(N^{-1})$ condition is well-studied in the Sobolev case  (see \cite{KoMa}  and references given there, \cite{Kl}). We study this condition for functions of bounded variation in the third section. 
 The proof of Theorem \ref{homos} is given in the fourth and fifth section. We also prove that it is enough to test $f$ on sets of finite perimeter.
\prt{Theorem}
\begin{proclaim}\label{neces3}
Let $\Omega_1, \Omega_2$ be open subsets of $\R^n$  and let $f$ 
 be a homeomorphism $\Omega_1\to \Omega_2$. Then the following conditions are equivalent:
\begin{enumerate}
\item There is a constant $K>0$ such that $P(f^{-1}(A),\Omega_1)\leq K P(A,\Omega_2)$.
\item The function $f$ has locally bounded variation and there exists a constant $K>0$ such that \eqref{klic1} holds.
\end{enumerate} 
\end{proclaim}

Actually we prove more general statements of the theorems. We allow $f$ to fail to be a homeomorphism. Our mapping will be a general  mapping of bounded variation (its multiplicity can be unbounded) with no jump part and satisfying \eqref{klic1} for some good representative of $f$.

\section{Preliminaries}
We use the usual convention that $C$ denotes a generic positive constant whose exact value may change from line to line.  We denote by $\Ln$ the Lebesgue measure. The symbol $\nabla u(x)$ denotes the classical gradient of $u$ in $x$. By $Du$ we denote the distributional derivative.

Let $\Omega$ be an open subset of $\R^n$. We write $G\subset\subset \Omega$ if the closure $\overline{G}$ is compact and  $\overline{G}\subset \Omega$.  A function $u\in L^1(\Omega)$ whose partial derivatives in the sense
of distributions are signed measures with finite total variation in $\Omega$ is called a function of bounded variation.  The vector space of functions of bounded variation is denoted by $BV(\Omega)$. We write  $u\in BV(\Omega,\R^d)$ if $u_i \in BV(\Omega)$ for all $i\in\{1,\ldots,d\}$.

If $u\in BV(\Omega,\R^d)$, the total variation of the measure $Du$ is defined by
$$|Du|(E)=\sup\Bigl\{\sum_{i=1}^m \int_E u_i \operatorname{div} \phi_i \,d\Ln:\phi \in C_c^1(\Omega, \R^{d\times n}), |\phi(x)|\leq 1 \text{ for } x\in \Omega\Bigr\}<\infty.$$
 We write $u\in BV_{\loc}(\Omega,\rn)$ if for all $x\in \Omega$ there is a ball $B\ni x$ such that $u\in BV(B,\rn)$.

 Propositions 3.6 and 3.13 in \cite{AFP} give us a simple characterization of BV functions
\prt{Theorem}
\begin{proclaim}\label{bvchar}
Suppose that $\Omega\subset \Rn$ is open and $u\in L^1 (\Omega)$. Then $u\in BV(\Omega)$ if and only if there is a sequence $u_k\in W^{1,1}(\Omega)$ such that $u_k \to u$ in $L^1$ and $\sup_k \|D u_k\|_{L^1}<\infty$.

Moreover,
$$|Du|(\Omega)= \inf \Bigl\{\sup_k \|D u_k\|_{L^1(\Omega)}; u_k\in L^1(\Omega), u_k \to u \text{ in }L^1(\Omega) \Bigr\}.$$
\end{proclaim}

\prt{Theorem}
\begin{proclaim}\label{weak}
Suppose that $\Omega\subset \Rn$ is open,  $u_k\in BV(\Omega)$   and there is $u\in L^1(\Omega)$ such that $u_k\rightarrow u$ in $L^1(\Omega)$ and $\sup_k |Du_k|(\Omega)<\infty$. Then $u$ belongs to $BV(\Omega)$ and $u_k$ weakly* converges to $u$ in $BV(\Omega)$.
\end{proclaim}

We say that $E\subset \Omega$ has finite perimeter if the characteristic function $\chi_E$ belongs to $BV(\Omega)$ and we set $$P(E,\Omega)=|D\chi_E|(\Omega).$$

The following lemma gives us a connection between functions of bounded variation and sets of finite perimeter. (Theorem 3.39 in \cite{AFP})
\begin{lemma}\label{perimetr}
Suppose that $\Omega\subset \R^n$ is open and $u\in L^1_{\loc}(\Omega)$. Then
\eqn{445}
$$|Du|(\Omega)=\int_{-\infty}^\infty P(\{x:u(x)>t\},\Omega)\,dt$$
\end{lemma}
We say that the approximate limit of $f\in L^1(\Omega,\R^n)$ exists at $x\in \Omega$ if there is $z\in \rn$ such that 
$$\lim_{r\to 0_+} \frac{1}{\Ln(B(x,r))} \int_{B(x,r)}|f(y)-z|\,dy=0.$$
We write $z=\applim_{y\to x} f(y)$. If $f$ is integrable then the set $S_f$ where the limit does not exists is  $\Ln$-negligible and Borel and 
$\tilde f=\applim f$ is Borel measurable on $\Omega\setminus S_f$. (See Proposition 3.66 in \cite{AFP}.)

Let us note that slightly weaker definitions of approximate limits are available in literature. For instance in \cite{Fe} $z\in \R^m$ is called the approximate limit of measurable function $f:\Omega \to \R^m$ at $x\in \Omega$ if all the sets
$$E_\epsilon=\{y\in \Omega:|u(y)-z|>\epsilon\}$$ 
has density 0 in $x$.
In our paper we follow the notation from \cite{AFP}. See the discussion which follows after Proposition 3.64 in \cite{AFP} to find differences between these definitions.

The main tool is the analogy of the chain rule for the composition of a smooth function and a function of bounded variation, see \cite{AM} or Theorem 3.96 in \cite{AFP}.
\prt{Theorem}
\begin{proclaim}\label{dercom}
Suppose that $\Omega\subset \Rn$ is open, $f\in BV (\Omega, \rn)$ and $u\in C^1(\rn,\R^k)$. Then the composition $u\circ f$ belongs to $BV(\Omega)$  and
$$D(u\circ f)= \nabla u\circ f \cdot D^a f \Ln +\nabla u\circ \tilde f\cdot D^c f+ [u(f^+)-u(f^-)]\otimes \nu_f \H^{n-1}|_J,$$
where 
$$Df=D^a f \Ln +D^c f+\nu_f \H^{n-1}|_J$$
is the usual decomposition of $Df$ in its absolutely continuous part $D^a f$ with respect to  the  Lebesgue measure $\Ln$,  its  Cantor  part $D^c u$ and  its  jump  part,  which  is represented by the restriction of the $(n-1)$ dimensional Hausdorff measure to the jump set $J$.  Moreover, $\nu_f$ denotes the measure theoretical unit normal to $J$, $\tilde f$ is the approximate limit and $f^+$, $f^-$ are the approximate limits from both sides of $J$.
\end{proclaim}
We will work only with functions which have no jump part, i.e. $J=\emptyset$. In that case we have 
$$D(u\circ f)=\nabla u \circ \tilde f \cdot Df.$$

\subsection{Basic properties of measures}
If $u$ is a  $\mu$-measurable function and $E$ is a $\mu$-measurable set then we denote by $\int_E u\, d\mu$ (or $\int_E u(x)\, d\mu(x)$ if we want to emphasize the variable) the integral of $u$ over $E$ with respect to the measure $\mu$. Instead of $d\Ln(x)$ we write shortly $dx$.

Given measure spaces $(X, \mathcal A)$ and $(Y, \mathcal B)$, a measurable mapping $f : X \to Y$ and a measure $\mu : \mathcal A \to [0, \infty]$, the image of  $\mu$ is defined to be the measure $f(\mu) : \mathcal{B} \to [0, \infty]$ given by
$$(f (\mu)) (A) = \mu \left( f^{-1} (A) \right) \text{ for } A \in \mathcal{B}.$$
Sometimes  $f (\mu)$ is called the pushforward of $\mu$.
\prt{Theorem}
\begin{proclaim}\label{chavar}
Let $X,Y, f, \mu$ be as above and $g:Y\to \R^n$ then we have that 
\eqn{obrmir}
$$\int_{Y} g \, d(f(\mu)) = \int_{X} g \circ f \, d\mu,$$ 
whenever one of the integrals is well-defined.
\end{proclaim}

Let $\mu, \nu$ be measures defined on the same $\sigma$-algebra $\mathcal A$ of the space $X$. We say that $\mu$  is
\begin{itemize}
\item  absolute continuous with respect to $\nu$ if 
$$|\nu|(A)=0\Rightarrow |\mu|(A)=0.$$
\item singular with respect to $\nu$ if there are $X_a, X_s\in \mathcal A$ such that $X=X_a\cup X_s$ and
$$|\nu|(X_s)=0=|\mu|(X_a).$$
We set $\sptt \nu=X_s$.
\end{itemize}
For each pair of non-negative measures $\mu$ and $\nu$ on the same $\sigma$-algebra $\mathcal A$ we can find a decomposition $\mu=\mu^a+\mu^s$ such that
$\mu^a$ is absolute continuous with respect to $\nu$ and $\mu^s$, $\nu$ are singular.
\prt{Theorem}
\begin{proclaim}[Radon-Nikodym]\label{radnik}
Let $\mu$ be a non-negative Borel measure on $\rn$ and set 
$$\frac{d\mu}{d\Ln}(x)=\lim_{r\to 0_+} \frac{\mu(B(x,r))}{\Ln(B(x,r))}.$$
Then $\frac{d\mu}{d\Ln}$ exists $\Ln$-a.e., $\frac{d\mu}{d\Ln}(x)$ is $\Ln$-measurable and 
$$\int_A \frac{d\mu}{d\Ln}(x) \, dx \leq \mu(A) \text{ for all Borel sets } A\subset G.$$
Moreover, if $\mu$ is absolute continuous with respect to $\Ln$ then the above inequality holds as equality.
\end{proclaim}

\section{Lusin $(N^{-1})$ condition}
In this section we generalize the result of P. Koskela and J. Mal\'y. In \cite{KoMa} they proved  our Theorem \ref{lusin} in  the special case when $f$ is a Sobolev mapping.

The following  lemma will be useful. See \cite[Lemma 2.1]{KoMa}.
\begin{lemma}\label{pravd}
There is a constant $\tau=\tau(n)$ with the following property: For each atomless probability Borel measure $\mu$ on $\rn$ there is a point $y\in \rn$ and a radius $R>0$ such that
$$\mu(B(y,2R))\geq \tau \text{ and } \mu(\rn \setminus B(y,3R))\geq \tau.$$
\end{lemma}

\prt{Theorem}
\begin{proclaim}\label{lusin}
Let $\Omega\subset \R^n$ be connected open set, $f\in BV(\Omega, \rn)$ have no jump part. 
Suppose that
\eqn{key789}
$$|Df|(\tilde f^{-1} (A)) \leq \int_A K(y) \,dy \text{ for all Borel set } A\subset \rn,$$ where $K(y)\in L^{p'}$ for some $p\in [1,n]$, $p'=\frac{p}{p-1}$ . If $f$ is not constant then $f$ satisfies Lusin $(N^{-1})$ condition, i.e. for any set $E\subset\rn$ we have
$$\Ln(E) = 0 \Rightarrow \Ln(f^{-1}(E)) = 0 .$$
\end{proclaim}
\begin{proof}
\textsl{ Step 1.}
Without loss of generality we may assume that $K$ is a Borel function and if $p=1$ we have  $K(y)= \operatorname{ess sup}_{z\in\R^n} K(z)$ for all $y\in\R^n$.

  We first prove an auxiliary estimate. With the help of  Theorem \ref{chavar} and the fact that for the image of measure $Df$ we have
	$$\tilde f(|Df|)(A)=|Df|(\tilde f^{-1} (A))\leq \int_A K(y) \,dy$$
	we obtain for each nonnegative Borel measurable  function $g$ and Borel set $A\subset \rn$ that
\eqn{formula}
$$\int_{f^{-1}(A)} g(\tilde{f}(x)) \,d|Df|(x)= \int_{A} g(y) \, d\tilde{f}(|Df|)(y)\leq \int_{A} g(y) K(y)\, dy.$$
 Note that the set $$N=\{x\in \Omega:K(\tilde f(x))=0\}=\tilde f^{-1}(\{y\in \rn:K(y)=0\})$$ has $|Df|$-measure zero.
Let $E\subset \Omega$ be a measurable set. Consider a smooth function $u$ with a compact support in $\rn$. 

With the help of Theorem \ref{dercom}, H\"older inequality and \eqref{formula} we can estimate
\eqn{klic58}
$$\begin{aligned}
 |D(u\circ f)&|(E)\leq \int_{E} |\nabla u(\tilde f(x))| \,d|Df|(x)\\
&= \int_{E\setminus N} |\nabla u(\tilde f(x))| K(\tilde f(x))^{-\frac{1}{n}} K(\tilde f(x))^{\frac{1}{n}} \,d|Df|(x) \\
&\leq \left(\int_{E\setminus N} |\nabla u(\tilde f(x))|^n K(\tilde f(x))^{-1} \,d|Df|(x)\right)^{1/n} \left(\int_{E} K(\tilde f(x))^{\frac{1}{n-1}} \,d|Df|(x)\right)^{1/n'}\\
&\leq \left(\int_{\rn} |\nabla u(y)|^n \,dy\right)^{1/n}\left(\int_E K(\tilde f(x))^\frac{1}{n-1} \,d|Df|(x)\right)^{1/n'}.
\end{aligned}$$

\textsl{ Step 2.} 
 We claim that
\eqn{odh23}
$$y_0 \in \rn  \Rightarrow  \Ln(f^{ -1}(\{y_0\})) = 0.$$

For this, consider an arbitrary ball $B\subset\subset \Omega$ and $y_0 \in \rn$. Suppose that $f$
differs from $y_0$ on a set of positive measure in $B$. Then there is $R > 0$ such
that
\eqn{odh24}
$$\kappa :=\Ln(B \setminus  f^{ -1}(B(y_0, R))) > 0.$$

Since singletons have zero $n$-capacity, given $\epsilon > 0$ there is a smooth function
$u$ on $\rn$ such that
$$\spt u \subset B(y_0, R),\ u(y_0) = 1 \text{ and } \int_{\rn} |\nabla u|^n \,d\Ln < \epsilon^n.$$
Then
$$\min \{\Ln(B \cap f^{ -1}(\{y_0\})), \kappa\}\leq  C r |D(u \circ f )|(B) .$$ 

For this we used the well-known trick
\eqn{od67}
$$1/2  \min\{\Ln(B \cap \{v \leq 0\}), \Ln(B \cap \{v \geq 1\})\} \leq  \inf_{c\in\R}\int_B|v - c| \,d\Ln 
\leq C r |Dv|(B),$$
based on the Poincare inequality, where the hypothesis is that $v \in BV.$
Note that from \eqref{formula} it follows for $p>1$ that 
$$\int_{\Omega} K(\tilde f(x))^\frac{1}{p-1} \,d|Df|(x)\leq \int_{\rn} K(y)^{p'}\,dy<\infty.$$
We know that $|Df|$ is a finite measure because $f\in BV(\Omega)$.  Hence $K(\tilde f(x))^\frac{1}{n-1}\in L^1(\Omega,|Df|)$. Trivially this relation holds even for the case when $p=1$.

Together with \eqref{klic58} we obtain
$$\begin{aligned}
\min\{|B \cap f^{ -1}(\{y_0\})|, \kappa\} &\leq  \left(\int_{\rn} |\nabla u|^n \,d\Ln\right)^{1/n} \left(\int_{B} K(\tilde f(x))^{\frac{1}{n-1}} \,d|Df|(x)\right)^{1/n'}\\
&\leq C \epsilon \left(\int_{B} K(\tilde f (x))^{\frac{1}{n-1}} \,d|Df|(x)\right)^{1/n'}.
\end{aligned}$$
Letting $\epsilon \to 0$ and using \eqref{odh24} we obtain that $\Ln(B \cap f^{ -1}(\{y_0\})) = 0$ whenever
$f$ differs from $y_0$ on a set of positive measure in $B$. Hence \ref{odh23} follows by taking the connectedness of $\Omega$ and the assumption that $f$ is not constant into account.

\textsl{ Step 3.} 
 Let us prove that there is some $c>0$ such that 
\eqn{radnyk}
$$\lim_{r\to 0_+}\frac{\int_{B(x_0,r)}K(\tilde f(x))^{\frac{1}{n-1}}\,d|Df|(x)}{\Ln(B(x_0,r))}>c$$ for a.e. $x_0$ in $\Omega$. 
Fix a ball $B(x_0,r)\subset\subset \Omega$. Consider the Borel measure $\mu$ defined by
$$\mu(A)=\frac{\Ln(B(x_0,r)\cap f^{-1}(A))}{\Ln(B(x_0,r))},\ A\subset \rn.$$
From Step 2 we know that $\mu$ does not have atoms. By Lemma \eqref{pravd} we find a point $y\in \rn$ and a radius $R>0$ such that
\eqn{procento2}
$$\mu(B(y,2R))\geq \tau \text{ and } \mu(\rn \setminus B(y,3R))\geq \tau.$$
Find a smooth function $u$ such that 
$$ u(x)=1 \text{ on } B(y,2R),\ u=0
\text{ outside } B(y,3R) \text{ and } \int_{\rn}|\nabla u|^n \,d\Ln\leq C(n).$$

The function $v=u\circ f$ belongs to $BV(\Omega)$ and we have
\eqn{prvak}
$$\begin{aligned}
 \frac{\Ln(B(x_0,r)\cap \{v=1\})}{\Ln(B(x_0,r))}&\geq \frac{\Ln(B(x_0,r)\cap f^{-1}(B(y,2R)))}{\Ln(B(x_0,r))}=\mu(B(y,2R))\geq \tau,\\
\frac{\Ln(B(x_0,r)\cap \{v=0\})}{\Ln(B(x_0,r))}&\geq \frac{\Ln(B(x_0,r)\setminus f^{-1}(B(y,3R)))}{\Ln(B(x_0,r))}=\mu(\rn \setminus B(y,3R))\geq \tau.
\end{aligned}$$

 By \eqref{prvak}
 , \eqref{od67} and \eqref{klic58}
  we have
\eqn{16}
$$\begin{aligned}
1&\leq C r^{1-n}  |D(u\circ f)|(B(x_0,r)) \leq C r^{1-n}  \left(\int_{B(x_0,r)} K(\tilde f(x))^{\frac{1}{n-1}} \,d|Df|(x)\right)^{1/n'}\\
&=C \left(\frac{1}{|B(x_0,r)|}\int_{B(x_0,r)} K(\tilde f(x))^{\frac{1}{n-1}} \,d|Df|(x)\right)^{1/n'}.
\end{aligned}$$

 \textsl{ Step 4.}
 From Step 3 we know that the Radon-Nikodym derivative of the measure $\nu=K(\tilde f(x))^{\frac{1}{n-1}} |Df|$ with respect to $\Ln$ is greater than some $c>0$. Let $E$ be an arbitrary set of measure zero. Take $\tilde{E}\supset E$ a Borel set of measure zero. 
It follows from Radon-Nikodym theorem \ref{radnik} and \eqref{formula} that
$$\begin{aligned}
c\Ln(f^{-1}(\tilde{E}))&= \int_{f^{-1}(\tilde{E})}c \,d\Ln\leq  \int_{f^{-1}(\tilde{E})} \frac{d\nu}{d\Ln} \,d\Ln
\\ &\leq \int_{f^{-1}(\tilde{E})} K(\tilde f(x))^{\frac{1}{n-1}} \,d|Df|(x)\leq \int_{\tilde{E}} K(y)^{\frac{n}{n-1}} \,dy=0.
\end{aligned}$$
Hence $f^{-1}(E)$ is a subset of a set of measure zero and it has measure zero.
\end{proof}

\prt{Theorem}
\begin{proclaim}\label{l1}
Let $f$ satisfy the assumption of Theorem \ref{lusin} for $p=1$ (i.e. the function $K$ is in $L^\infty(\R^n)$). Then the operator $T_f$ defined by $T_f(u)(x)=u(f(x))$ maps  $L^1(\rn)$ to $L^1(\Omega)$ boundedly.
\end{proclaim}
\begin{proof}
First note that from Theorem we know tha DODELAT

 Without loss of generality we may assume that $K(y)\leq K$ everywhere and then we obtain by \eqref{16} that
$$\begin{aligned}
1&\leq C K^{n} \left(\frac{1}{\Ln(B(x_0,r))}\int_{B(x_0,r)} \,d|Df|\right)^{1/n'}.
\end{aligned}$$
Thus we proved that the Radon-Nikodym derivative of measure $\nu=|Df|$ with respect to $\Ln$ is greater than some $c>0$. It follows from Radon-Nikodym theorem \ref{radnik} and \eqref{formula} that
\eqn{l11}
$$\begin{aligned}
c \int_{\Omega} |u\circ f| \,d\Ln&\leq  \int_{\Omega} |u\circ \tilde f| \frac{d|Df|}{d\Ln} \,d\Ln
\\ &\leq \int_{\Omega} |u\circ \tilde f| \,d|Df|\leq \int_{\Rn} |u(y)| K(y) \,dy\leq K \int_{\Rn} |u(y)|  \,dy.
\end{aligned}$$
\end{proof}

\begin{rmk}
Analogously to Theorem \ref{l1} it is possible to show that for such mapping $f$ its operator $T_f$ maps any rearrangement invariant space $X(\rn)$ to $X(\Omega)$ and $\|u\circ f\|_{X(\Omega)}\leq c\|u\|_{X(\rn)}$. To get the sufficient estimate on the level set use \eqref{l11} on $u=\chi_{\{|u|\geq \alpha\}}$.
\end{rmk}

The conditions on $f$ in Theorem \ref{lusin} are sharp. 
For all $p>n$ there is a Sobolev self-homeomorphism of $(0,1)^n$ such that $K(y) \in L^{p'}$ but Lusin $(N^{-1})$ condition fails. Indeed, in \cite{Kl} we constructed a homeomorphism of finite distortion such that $|Df(x)|^{p}\leq L(x) J_f(x)$ a.e. with $L(x)\in L^\infty$, but Lusin $(N^{-1})$ condition fails. Let us show that \eqref{key789} from Theorem \ref{lusin} is satisfied for $K(y)=L(f^{-1}(y))^{\frac 1p} J_f(f^{-1}(y))^{\frac{1-p}p}$.  

Denote by $N$  a set of measure zero such that $f|_{(0,1)^n\setminus N}$ satisfies the Lusin $(N)$ condition. Set $Z=\{x:J_f(x)=0\text{ or does not exist}\}$. Then with the help  of Area formula (see \cite[Theorem 2]{H1}) we easily obtain
$$\begin{aligned}
|Df|(E)&\leq \int_{E} L(x)^{\frac 1p} J_f^{\frac 1p}(x) \,dx = \int_{E\setminus (Z\cup N)} L(x)^{\frac 1p} J_f^{\frac 1p}(x)\,dx\\
&=\int_{f(E\setminus (Z\cup N))} L(f^{-1}(y))^{\frac 1p} J_f(f^{-1}(y))^{\frac{1-p}p}\,dy.
\end{aligned}$$ 
It remains to show that $L(f^{-1}(y))^{\frac 1p} J_f(f^{-1}(y))^{\frac{1-p}p}\chi_{(0,1)^n\setminus f(Z\cup N)}$ is in $L^{p'}$. This follows since
$$\begin{aligned}
\int_{(0,1)^n} &\left(L(f^{-1}(y))^{\frac 1p} J_f(f^{-1}(y))^{\frac{1-p}p}\chi_{(0,1)^n\setminus f(Z\cup N)}\right)^{p'}\,dy\\
&=\int_{(0,1)^n \setminus f(Z\cup N)}  L(f^{-1}(y))^{p'-1} J_f(f^{-1}(y))^{-1}\,dy= \int_{(0,1)^n} L(x)^{p'-1}\,dx<\infty.
\end{aligned}$$

Surprisingly it is not enough to control by the absolute continuous part of the derivative $Df$. Indeed, it is possible to construct a homeomorphism $f$  such that for any constant $K\in \R$  we have
\eqn{ac}
$$|D^af|(x)\leq K J_f(x) \text{ for a.e. }x$$
 and Lusin $(N^{-1})$ condition fails. In \cite{He} we can find a Sobolev homeomorphism $g$ of $(0,1)^n$ such that $J_g=0$ a.e. and $|Dg|\in L^{n-1}$. The homeomorphism $g$ maps a set of full measure into  a set of measure zero and a set of measure zero into  a set of full measure. Let us show that $f=g^{-1}$  satisfies $|D^af|(x)=0$ and hence also  \eqref{ac}. 

It follows from Lemma 4.3 in \cite{CHM} and Theorem 3.8 \cite{Mi} that  $f=g^{-1}\in BV((0,1)^n,(0,1)^n)$ and 
$$|Df|(f^{-1}(G))\leq C\int_{G} |\operatorname{adj} Dg| \,d\Ln$$ holds for every open $G\subset (0,1)^n$ where $C$ depends only on $n$. Hence it also holds for each Borel set $A$ and we have 
$$|Df|(f^{-1}(A))\leq  C\int_{A} |\operatorname{adj} Dg| \,d\Ln.$$

 Denote by $N$ a Borel set $N\subset (0,1)^n$ such that $\Ln(N)=0$ and $g(N)=f^{-1}(N)=\Ln((0,1)^n)$. Then
$$|D^a f|((0,1)^n)= |D^a f|(f^{-1}(N))\leq |Df|(f^{-1}(N))\leq C \int_N |\operatorname{adj} Dg|^{n-1} \,d\Ln=0.$$
Thus $|D^af|=0$ a.e. and the inequality \eqref{ac} trivially  holds.

\section{Sufficient condition}

\prt{Theorem}
\begin{proclaim}\label{slozen3}
Let $\Omega_1, \Omega_2$ be open subsets of $\R^n$  and let $f\in BV_{\loc}(\Omega_1,\Omega_2)$ have no jump part.  
 Suppose that $f$ is not constant on any component of $\Omega$ and there is a constant $K>0$ such that
\eqn{klic55} 
$$|Df|(\tilde f^{-1}(A))\leq K \Ln (A)\text{ for all Borel sets } A\subset\Omega_2.$$
Then the operator $T_f(u)(x)=u(f(x))$ maps functions from  $BV(\Omega_2)$  into $BV(\Omega_1)$ and
\eqn{spojitost}
$$|D(u\circ f)|(\Omega_1)\leq K |Du|(\Omega_2).$$
\end{proclaim}

\begin{proof}
Suppose that $u\in BV(\Omega_2)$ . Let be $u_k$ an approximation of $u$ from Theorem \ref{bvchar} and $G\subset\subset \Omega_1$ be an open set.
We prove that $u_k\circ f$ is a good approximation of $u\circ f$ on $G$. 

  It follows from Theorem  \ref{l1} that $u\circ f \in L^1(\Omega)$  and due to $\eqref{l11}$ we obtain
$$\|u_k\circ f -u\circ f\|_{L^1(\Omega_1)}\leq C \|u_k-u\|_{L^1(\Omega_2)}.$$ Thus  $u_k\circ f \to u\circ f$ in $L^1(G)$. Let us note that Theorem \ref{lusin} is key for us. It  gives us validity of Lusin $(N^{-1})$ condition for the function $f$ and hence the composition  $u\circ f$ is a well-defined function.

By Theorem \ref{dercom} we have that $u_k\circ f$ belongs to $BV(G)$ and $D(u_k\circ f)(x)=\nabla u_k\bigr(\tilde{f}(x)\bigl)\cdot Df(x) $. As in \eqref{formula}   we can with the help of Theorem \ref{chavar} and the fact that $\tilde f(|Df|)(A)\leq K \Ln(A)$  estimate
$$
\begin{aligned}
|D(u_k\circ f)|(G)&\leq \int_{G}\left|\nabla u_k\bigr(\tilde f(x)\bigl)\right| \,d |Df|(x)
 \leq K \int_{\Omega_2}|\nabla u_k| \, d\Ln.
\end{aligned}$$ Lemma \ref{weak} gives us that $u\circ f$ has bounded variation on $G$.
Moreover,  using semi-continuity of the variation we obtain
$$
\begin{aligned}
|D(u\circ f)|(G)&\leq \inf \Bigl\{\sup_k \|D v_k\|_{L^1}:v_k\in L^1(G),v_k {\rightarrow} u\circ f \text{ in }L^1(G) \Bigr\}\\
&\leq \inf \Bigl\{\sup_k \|D (u_k\circ f)\|_{L^1}:u_k\in C^\infty,u_k {\rightarrow} u \text{ in } L^1(\Omega_2)\Bigr\}\\
&\leq K \inf \Bigl\{\sup_k \|D u_k\|_{L^1}:u_k\in C^\infty,u_k\rightarrow u \text{ in }L^1(\Omega_2)\Bigr\}\\
&= K |Du|(\Omega_2).
\end{aligned}$$
To prove \eqref{spojitost} find open sets $G_k\subset\subset \Omega$ such that $G_k\subset G_{k+1}$ and $\Omega_1=\bigcup_{k=1}^{\infty} G_k$ then 
$$|D(u\circ f)|(\Omega_1)=\lim_{k\to \infty} |D(u\circ f)|(G_k)\leq K |Du|(\Omega_2).$$

\end{proof}
In the case when $f$ is constant on some component $G$ of $\Omega$ the composition $u\circ f$ may fail to be well-defined. If we take a representative of $u$ such that $\tilde u(x)=0$ for all $x$ such that there is a component $G$ of $\Omega$ satisfying $f(G)=\{x\}$ then for this representative we have $\tilde u\circ f \in BV(\Omega_1)$ and  \eqref{spojitost} again holds.

By applying Theorem \ref{slozen3}  on characteristic functions of sets we easily obtain the following corollary.
\begin{corollary}\label{sufper}
Let $f$ satisfy the assumptions of Theorem \ref{slozen3}. Then for any set of finite perimeter $E\subset \Omega_2$  the preimage $f^{-1}(E)$ is a set of finite perimeter in $\Omega_1$ and $$P(f^{-1}(E),\Omega_1)\leq K P(E,\Omega_2).$$
\end{corollary}

\begin{rmk}
The condition \eqref{klic55} can be rewritten as
\eqn{bobabobek}
$$\int_{\tilde f^{-1} (A)} |D^a f|\,d\Ln + \int_{\tilde f^{-1} (A)} \,d|D^c f| \leq K\Ln(A) ,$$
which is equivalent to existence of constants $C_1,C_2\in \R$ such that 
\eqn{blb}
$$\int_{\tilde f^{-1} (A)} |D^a f|\,d\Ln \leq C_1 \Ln(A)$$
 and 
\eqn{bob}
$$ \int_{\tilde f^{-1} (A)} \,d|D^c f|\leq C_2 \Ln(A).$$
The second condition \eqref{bob} implies that $|D^c f|(\tilde f^{-1}(A))=0$ whenever $A\subset \Omega_2$ has measure zero.
\end{rmk}

\prt{Lemma}
\begin{proclaim}
Let $f$ belong to $BV(\Omega,\rn)$, have no jump part and satisfy $f(z)=\applim_{x\to z} f(x)$ whenever $z \in \sptt |D^c f|$ and the limit exists. Then $\eqref{klic55}$ holds if and only if
\eqn{bobabobek*}
$$|D f|(f^{-1}(A))\leq K\Ln(A) \text{ for all Borel sets } A\subset\Omega_2.$$
\end{proclaim}
\begin{proof}
Because $f^{-1}(A)$ and $\tilde f^{-1}(A)$ differ only by a set of $\Ln$ measure  zero we have 
\eqn{bob*} 
$$\int_{ f^{-1} (A)} |D^a f|\,d\Ln =\int_{ \tilde f^{-1} (A)} |D^a f|\,d\Ln.$$
For the second part we will use facts which can be found in Chapter 3 in \cite{AFP}. The set $S_f$ where the approximate limit does not exists is $\mathcal H^{n-1}$-negligible \cite[Theorem 3.76]{AFP}.
Thus  $\tilde f^{-1} (A)\cap \sptt |D^c f|$ and  $f^{-1} (A)\cap \sptt |D^c f|$ are equal up to a set of  $\mathcal H^{n-1}$-Hausdorff measure zero.  Because $Du$ does not see  sets of  $\mathcal H^{n-1}$ measure zero \cite[Lemma 3.76]{AFP}, we have 

\eqn{blb*}
$$ \int_{ f^{-1} (A)\cap \sptt |D^c f|} \,d|D^c f|= \int_{\tilde f^{-1} (A)\cap \sptt |D^c f|} \,d|D^c f|.$$
These two equalities together with \eqref{bobabobek} give us \eqref{bobabobek*}.
\end{proof}
Thus we may take in Theorem \ref{slozen3} the natural representative satisfying $f(x)=\lim_{r\to 0_+} \frac{1}{\Ln(B(x,r))}\int_{B(x,r)} f(z)\,dz$ and demand the condition $\eqref{bobabobek}$.

\prt{Lemma}
\begin{proclaim}
Assume that $f$ is a homeomorphism of bounded variation. Then the inequality \eqref{blb} is equivalent to 
\eqn{1wconf}
$$|D^a f(x)|\leq C_1 |J_f|(x) \text{ for a.e. }x \in \Omega_1.$$
\end{proclaim}
\begin{proof}
It easily follows from \eqref{1wconf} and Area formula (see \cite[Theorem 2]{H1}) that we have 
$$\int_{f^{-1} (A)} |D^a f|\, d\Ln\leq C_1 \int_{f^{-1} (A)} |J_f| \, d\Ln\leq C_1 \Ln(A).$$
To prove the second implication let us assume that $x$ is a Lebesgue point of $J_f$ and $D^a_f$. Find a Borel set $N$ of measure zero such that $f|_{\Omega_1\setminus N}$ satisfies Lusin $(N)$ condition. It follows by \eqref{blb} that
$$\int_{B(x,r)} |D^af| \, d\Ln= \int_{B(x,r)\setminus N} |D^af|\, d\Ln \leq C_1\left|f\left(B(x,r)\setminus N\right)\right|= C_1\int_{B(x,r)}|J_f| \, d\Ln.$$ 
By dividing the both sides by $\Ln(B(x,r))$ and sending $r\to 0$ we get \eqref{1wconf}.
\end{proof}

If we assume that $f$ is a Sobolev homeomorphism then $D^c f=0$. 
\prt{Corollary}
\begin{proclaim}
If $f$ is a homeomorphism in $W_{\loc} ^{1,1}(\Omega_1,\R^n)$, then \eqref{klic55} is equivalent to
\eqn{1conf}
$$|Df(x)|\leq K |J_f|(x) \text{ for a.e. }x \in \Omega_1.$$
\end{proclaim}

The simplest way to obtain the condition \eqref{klic55} is to check the integrability of the inverse.
\prt{Lemma}\begin{proclaim}\label{lip2}
Let $\Omega_1,\Omega_2 \subset \rn$ and let $f:\Omega_1\to \Omega_2$ be a mapping such that $f^{-1}$ is Lipschitz. Then \eqref{klic55} holds.
\end{proclaim}
\begin{proof}
It follows from Lemma 4.3 in \cite{CHM} and Theorem 3.8 in \cite{Mi} that  $f\in BV_{\loc}(\Omega_2,\Omega_1)$ and 
\eqn{odh11}
$$|Df|(f^{-1}(G))\leq C\int_{G} |\operatorname{adj} D(f^{-1})| \,d\Ln,$$ where $C$ depends only on $n$. Hence \eqref{odh11} holds for all Borel sets and we have 
$$|Df|(f^{-1}(A))\leq  C\int_{A} |\operatorname{adj} D(f^{-1})|\, d\Ln\leq C\|D(f^{-1})\|^{n-1}_{L^\infty} \Ln(A).$$
\end{proof}
\begin{example}
There is a homeomorphism $f$ such that \eqref{klic55} holds but $f\notin W^{1,1}_{\loc}$.
\end{example}
\begin{proof}
 Consider the usual Cantor ternary function $u$ on the interval $(0,1)$. And set $g(x)=u(x)+x$. This function is continuous, increasing and fails to be absolutely continuous. Moreover, $g$ does not belong to $W^{1,1}_{\loc}$. On the other hand, the inverse function $g^{-1}$ is Lipschitz and maps $(0,2)$ homeomorphically onto $(0,1)$. If we set 
$$f(x_1,\ldots,x_n)=(g(x_1),x_2,\ldots, x_n)$$ then obviously $f$ fails to belong to $W^{1,1}_{\loc}((0,1)^n,\rn)$, and $f^{-1}$ is a  Lipschitz function. Due to Lemma \ref{lip2} the function $f$ satisfies \eqref{klic55}.
\end{proof}

In the special case when $n=2$ we obtain the equivalence in Lemma \ref{lip2}. 
\prt{Lemma}\begin{proclaim}
Let $\Omega_1,\Omega_2 \subset \R^2$ and $f:\Omega_1\to \Omega_2$ be a homeomorphism. Then $f^{-1}\in W^{1,\infty}(\Omega_2,\Omega_1)$ if and only if $f\in BV_{\loc}(\Omega_1,\Omega_2)$ and \eqref{klic55} holds.
\end{proclaim}
\begin{proof}
It remains to prove the second implication. Let $f\in BV_{\loc}(\Omega_1,\Omega_2)$. It follows from \cite{DS} 
that $f^{-1}$ is in $BV_{\loc}(\Omega_2,\Omega_1)$ and
$$|D(f^{-1}_1)|(\Omega_2)=|D_y f|(f^{-1}(\Omega_2)) \text{ and }|D(f^{-1}_2)|(\Omega_2)=|D_x f|(f^{-1}(\Omega_2)).$$
It holds for all open set $\Omega_2$ thus we have for all Borel sets $A\subset \Omega_2$
$$|D(f^{-1})|(A)\leq 2|Df|(f^{-1}(A)).$$ 
By combining with \eqref{klic55} we have $|D(f^{-1})|(A)\leq 2 K\Ln(A)$ and thus the measure $D(f^{-1})$ is absolutely continuous with respect to $\Ln$ and its Radon-Nikodym derivative with respect to $\Ln$ is in $L^{\infty}$.
\end{proof}

\begin{rmk}
 This is not true in higher dimensions. Let $n\geq 3$. There exists a Sobolev homeomorphism $f$ on $[-1,1]^n$ onto $[-1,1]^n$ such that \eqref{1conf} (thus also \eqref{klic55}) is satisfied but $f^{-1}$ is not even a Sobolev function.
\end{rmk}
This function is constructed in Example 6.3. in \cite{HeKoMa}. They construct a Sobolev homeomorphism which maps one Cantor set on another one and it is piecewise affine on the complement of the Cantor set. For arbitrary $\epsilon\in(0,n-2)$  choose  an big enough parameter $l\in \N$ such that  $\epsilon l > 2(n - 1 - \epsilon)$. Then  for their function $f$ the following holds $Df\approx k^l$ and $J_f\approx k^{l-1}k^{l(n-2)}k^{-1}$ on $A_k$, $k\in \N$, where $A_k$ are pairwise disjoint sets of positive measure whose union has the full measure of  $[-1,1]^n$. Then  $$\frac{|Df|}{J_f}\approx \frac{k^l}{k^{l-1}k^{l(n-2)}k^{-1}}= \frac{1}{k^{l(n-2)-2}} \stackrel{k\to \infty}{\rightarrow} 0.$$
Hence $f$ satisfies \eqref{1conf} (even stronger condition $|Df|^{n-1-\epsilon}\leq CJ_f$ for some $C>0$). But they show that $f^{-1}$ is continuous but not ACL.

\section{A necessary condition}

\prt{Theorem}
\begin{proclaim}\label{neces}
Let $\Omega_1, \Omega_2$ be open subsets of $\R^n$  and let $f\in BV_{\loc}(\Omega_1,\Omega_2)$ have no jump part and suppose that  the operator $T_f$ maps functions from $C_c^\infty(\Omega_2)$ into $BV_{\loc}(\Omega_1)$ and there is a constant $K\in \R$ such that for all $u\in C_c^\infty(\Omega_2)$ we have
\eqn{necesklic}
$$|D(u\circ f)|(\Omega_1)\leq K |Du|(\Omega_2).$$
 Then   for all Borel set $A\subset \Omega_2$  we have
\eqn{klic56}
$$|Df|(\tilde{f}^{-1}(A))\leq 16n K \Ln (A).$$
\end{proclaim}
\begin{proof}
We may assume that $f=\tilde f$. (We change $u\circ f$ only on a set of measure zero.) Take $A\subset \Omega_2$ a Borel set.
  Suppose that  $|D f|(f^{-1}(A))\neq 0$, otherwise there is nothing to prove. Let $t>0$ and $0<L< |Df|(f^{-1}(A))$ be  arbitrary real numbers and fix $i\in \{1,\ldots,n\}$ such that $$|D(f_i)|(f^{-1}(A))\geq \frac{1}{n} |Df|(f^{-1}(A))>\frac{1}{n} L .$$ 
	Find an open set $G\subset \Omega_2$ such that $A\subset G$ and
$\Ln(G)\leq \Ln(A)+t$. 
 Then $$A=\bigcup_k A_k=\bigcup_k \{x\in A\cap B(0,k): \operatorname{dist}(x,\partial G)\geq 1/k\}.$$
Choose $k\in \N$ big enough such that $$|Df|(f^{-1}(A_k))>\frac{1}{n} L.$$
Find  a cut-off function $\eta\in \C_c^{\infty}(\Omega_2)$ 
satisfying  
$$\spt\eta\subset G,\ 0\le \eta\le 1\text{ and }\eta=1 \text{ on }A_k.$$ 
Take $m$ such that $m\geq 8$ and
$\|\nabla \eta\|_{\infty}\leq m$. 
Choose $E$ among the sets
$$
\aligned
E^{\sin}&=
\{x\in f^{-1}(A_k):\ \cos^2 ( m^2  f_i(x))\geq \tfrac12\},\\
E^{\cos}&=
\{x\in f^{-1}(A_k):\ \sin^2 ( m^2  f_i(x))\geq \tfrac12\}
\endaligned
$$
such that
$$
|D(f_i)|(E)\geq \tfrac 12 |D(f_i)|(f^{-1}(A_k))
$$
and set 
$$
u (y)=
\begin{cases}
\frac{1}{m^2} \eta(y)\sin (m^2 y_i)&\text{ if }E=E^{\sin}
\\
\frac {1}{m^2} \eta(y)\cos (m^2 y_i)&\text{ if }E=E^{\cos}.
\end{cases}
$$
First consider $E=E^{\sin}$. Obviously $u\in C^\infty_c(\Omega_2)$ and 

\eqn{lips}
$$|\nabla u(y)|= |1/m^2\nabla \eta(y) \sin (m^2 y_i)+\eta(y) \cos(m^2 y_i) e_i|\leq 2\text{ for all } y\in \Omega_2.$$
By the  product rule from Theorem \ref{dercom} it easily follows 
$$
\begin{aligned}
|D(u\circ f)|(E)&= \int_E  \,d|D(u \circ f)|=\int_E  |\nabla u|(\tilde f(x)) \,d|D f_i|\\
&\geq \int_E \left(|\eta(f) \cos(m^2 f_i) e_i|-|1/m^2\nabla \eta(f) \sin (m^2 f_i)|\right)\, d|D f_i|\\
&\geq \int_E (\tfrac{1}{\sqrt{2}}-\tfrac{1}{m})\, d|D f_i| \geq \frac 1{4}\ |D f_i|(E)\\
&\geq \frac 18|Df_i|(f^{-1}(A_k))\geq \frac{1}{8n} L.
\end{aligned}$$

Thus together with \eqref{necesklic}, $\spt u\subset G$ and $|\nabla u|
\leq 2$ we estimate
$$\begin{aligned}
 L\leq  8n |D(u\circ f)|(\Omega_1)&\leq  8n\, K |Du|(\Omega_2)\leq 8n K\, 2\cdot\Ln(G)\leq 16Kn (\Ln(A)+t).
\end{aligned}$$
By taking supremum over all $L\leq |D f|(f^{-1}(A))$ and letting $t\to 0$ we obtain \eqref{klic56}.

The case when $E=E^{\cos}$ is analogous.

\end{proof}
The following corollary gives us that we may only assume that $f^{-1}$ maps sets of finite perimeter onto sets of finite perimeter.

\begin{corollary}\label{necper}
Let $\Omega_1, \Omega_2$ be open subsets of $\R^n$  and let $f\in BV_{\loc}(\Omega_1,\Omega_2)$ have no jump part. If for all sets of finite perimeter $E\subset \Omega_2$ sets $f^{-1}(E)$ have finite perimeter and
$$P(f^{-1}(E),\Omega_1)\leq K P(E,\Omega_2).$$
Then \eqref{klic56} holds.
\end{corollary}
\begin{proof} 
We show that the assumptions of Theorem \ref{neces} are satisfied. Let $u\in C^\infty_c(\Omega_2)$ then $u\circ f\in L^\infty(\Omega_1)$ and with the help of Lemma \ref{perimetr} we get

$$\begin{aligned}
|D(u\circ f)|(\Omega_1)&=\int_{-\infty}^\infty P(\{x:u(f(x))>t\},\Omega_1)\,dt\\
 &= \int_{-\infty}^\infty P(f^{-1}(\{y:u(y)>t\}),\Omega_1)\,dt\\
&\leq K \int_{-\infty}^\infty P(\{y:u(y)>t\},\Omega_2)\,dt\leq K |Du|(\Omega_2).
\end{aligned}$$
Thus  $u\circ f\in BV(\Omega_1)$ and \eqref{necesklic} holds.
\end{proof}

\begin{proof}[Proof of Theorem \ref{homos}]
The first part follows directly form Theorem \ref{slozen3}. Let us prove the second part.
First note that $f\in BV_{\loc}(\Omega_1,\Omega_2)$. To see it take an arbitrary ball $B\subset \subset \Omega$. Then $f(B)\subset\subset \Omega_2$ and hence we can find a smooth cutoff function $\Phi$ such that $\Phi=1$ on $f(B)$ and $\spt \Phi \subset \subset \Omega_2$. It follows that $u=e_i \Phi$, $i\in (1,\ldots,n)$ are suitable test function and $u\circ f= f_i$ on $B$. Thus each component $f_i$ belongs to  $BV_{\loc}(\Omega_1)$.

Suppose that \eqref{klic1} does not hold. Then there are Borel sets $G_k,$ $k\in \N$ such that 
\eqn{klic99}
$$|Df|(\tilde{f}^{-1}(G_k))>  k \Ln (G_k).$$
 
Because the Lebesgue measure is regular we may assume $G_k$ are open. Moreover, we may assume that  $|Df|(\tilde{f}^{-1}(G_k))<\infty$, otherwise we would replace $G_k$ by $G_k\cap \{x\in B(0,R), \operatorname{dist}(x,\partial \Omega_1)<1/R\}$ for some $R$ big enough. We claim that is possible to find pairwise disjoint open sets $G_k$ satisfying \eqref{klic99}. 

Let $l\in\N$, $G_k$ satisfies $\eqref{klic99}$,  $G_1,\ldots G_{l-1}$ are pairwise disjoint and 
$$\bigcup_{i=1}^{l-1}G_i\cap \bigcup_{i=l}^{\infty}G_i=\emptyset.$$ 
We describe how to construct $\tilde G_k$ which has properties of $G_k$ but additionally $\tilde G_l\cap \bigcup_{i=l+1}^{\infty}\tilde G_i=\emptyset$.
   Fix some $m\geq {l}/{\tau}$, where  $\tau$ is from Lemma \ref{pravd}.
  Due to  the  non-atomicity of the measure $|Df|$  we may use Lemma \ref{pravd} on the measure 
	$$\mu(A)=\frac{|Df|({f}^{-1}(A\cap G_{m}))}{|Df|({f}^{-1}(G_{m}))}$$
to find open sets $P_1=B(y,2R)$, $P_2=\rn \setminus \overline{B(y,11/4 R)}$, $R_1=B(y,10/4 R)$, $R_2=\rn \setminus \overline{B(y,9/4 R)}$ such that $$\mu(P_1),\ \mu(P_2)\geq \tau, \ P_1\cap R_2=\emptyset= P_2\cap R_1,\ R_1\cup R_2=\rn.$$
Then we obtain for all $i\in \N$ that
$$\begin{aligned}
|Df|(f^{-1}(G_i\cap R_1))&+|Df|(f^{-1}(G_i\cap R_2))\geq |Df|(f^{-1}(G_i))\\
&> i \Ln (G_i)\geq i/2\, \Ln (G_i\cap R_1)+i/2\ \Ln (G_i\cap R_2).
\end{aligned}$$
 Hence at least one of the sets 
$$\begin{aligned}
C&=\{i\in \N,i>m,\ |Df|(f^{-1}(G_i\cap R_1))> i/2  \, \Ln (G_i\cap R_1)\},\\
D&=\{i\in \N, i>m,\ |Df|(f^{-1}(G_i\cap R_2))> i/2  \, \Ln (G_i\cap R_2)\}
\end{aligned}$$
has to be infinite. First consider the case when $C$ is infinite.  Let $c_i, i\in \N$  be an increasing sequence containing all elements of  $C$.  Set $\tilde G_i=G_i$ for $i<l$,
$$\tilde G_l= G_{m}\cap P_2 \text{ and } G_i= G_{c_{2i}}\cap R_1 \text{ for } i>l.$$ 
Then obviously $\tilde G_i$, $i\in \{1,\ldots l\}$ are pairwise disjoint and $$\bigcup_{i=1}^{l}\tilde G_i\cap \bigcup_{i=l+1}^{\infty}\tilde G_i=\emptyset.$$ It remains to verify \eqref{klic99}. It follows that
$$|Df|({f}^{-1}(\tilde G_l))=  \mu(P_2) \cdot |Df|({f}^{-1}(G_m)) >  \tau m \Ln(G_{m})\geq l \Ln(\tilde G_{l})$$
and for all $i>l$ we have
$$|Df|({f}^{-1}(\tilde G_i))= |Df|\left({f}^{-1}(G_{c_{2i}}\cap R_1)\right)> 1/2 \cdot c_{2i} \Ln( G_{c_{2i}}\cap R_1)\geq i \Ln(\tilde G_{i}).$$ The case when $D$ is infinite is analogous. Thus we may iterate this construction to obtain $G_1,G_2,\dots$ pairwise disjoint.

Because $f$ has no jump part and \eqref{klic56} does not hold on $G_k$ with $K= \frac{1}{16n} k$ it follows from Theorem \ref{neces} that there are a $u_k\in C_c^\infty(G_k)$ such that
$$|D(u_k\circ f)|(\Omega_1)> \frac{1}{16n} k |Du_k|(G_k).$$
 Replace $u_k$ by its constant multiple to obtain $|Du_k|(G_k)=1$. Due to the fact that $|D|v||=|Dv|$ for any function $v$ of bounded variation we may assume that $\|u_k\|_{L^\infty} \leq 1$ (Otherwise we can  iterate replacing $u_k$ by function $\tilde{u_k}=||u_k|-1/2\|u_k\|_{L^\infty}|$, which has the same total variation of the distributional derivative and its maximum is half of the maximum of $u_k$.)
Set 
$$u=\sum_{k=1}^\infty \frac{1}{k^2} u_k.$$
Obviously $u\in C_0 \cap BV(\Omega_2)$ and
$$|Du\circ f|(\Omega_1)=\sum_{k=1}^\infty \frac{1}{k^2}|Du_k\circ f|(f^{-1}(G_k))=\sum_{k=1}^\infty Ck \frac{1}{k^2}=\infty.$$
\end{proof}
\begin{proof}[Proof of Theorem \ref{neces3}]
To prove first part we can follow the proof of Corollary \ref{necper} and instead of using  Theorem \ref{neces} we use Theorem \ref{homos}.

The second implication follows directly from Corollary \ref{sufper}.
\end{proof}

\end{document}